\documentclass[letterpaper, english]{amsart}

\title{Graded topological spaces}
\author{Clemens Koppensteiner}

\usepackage[utf8]{inputenc}
\usepackage[T1]{fontenc}
\usepackage{lmodern}
\usepackage{microtype}

\usepackage{amsmath,amsfonts,amssymb}
\usepackage{mathtools,extpfeil,xfrac}
\usepackage{amsthm,thmtools}

\usepackage{tikz-cd}

\usepackage{enumitem}
\setlist[enumerate,1]{label=(\roman*)}
\setlist[enumerate,2]{label=(\alph*)}
\setlist[enumerate,3]{label=(\Alph*)}
\setlist[enumerate,4]{label=\arabic*.}

\usepackage[
    backend=biber,
    style=alphabetic,
    maxnames=100,
    maxalphanames=100,
]{biblatex}

\DeclareLabelalphaTemplate{
    \labelelement{
        \field[final]{shorthand}
        \field{label}
        \field[varwidthlist,strside=left,names=4]{labelname}
    }
}

\DeclareFieldFormat{labelalpha}{#1}
\DeclareFieldFormat{extraalpha}{#1}

\newcommand{\stackcite}[1]{\cite[\href{http://stacks.math.columbia.edu/tag/#1}{Tag~#1}]{stacks-project}}

\usepackage{mparhack}  % correct placement of margin notes
\usepackage{todonotes}
\usepackage{csquotes}
\usepackage[overload]{textcase}
\usepackage{url}

\usepackage[colorlinks=false,unicode=true,bookmarksdepth=2]{hyperref}
\usepackage{bookmark}

\theoremstyle{plain}
\newtheorem{Theorem}{Theorem}[section]

\newtheorem{Proposition}[Theorem]{Proposition}
\newtheorem{Corollary}[Theorem]{Corollary}
\newtheorem{Lemma}[Theorem]{Lemma}

\theoremstyle{definition}
\newtheorem{Definition}[Theorem]{Definition}

\theoremstyle{remark}
\newtheorem{Remark}[Theorem]{Remark}

\newcommand\ZZ{\mathbb{Z}}

\newcommand\RR{\mathbb{R}}
\newcommand\CC{\mathbb{C}}

\newcommand\cat\mathbf
\newcommand\Hom{\operatorname{Hom}}

\newcommand\opcat[1]{{#1}^{\mathrm{op}}}    % opposite category

\newcommand\catMod[1]{#1\mbox{-}\cat{mod}}  % category of R-modules

\newcommand\supp{\operatorname{supp}}   % support of a sheaf
\newcommand\sheaf\mathcal
\newcommand\sheafHom{\underline{\Hom}}  % sheafy Hom

\newcommand\catPreSh[1]{\cat{PSh}(#1)}  % category of presheaves
\newcommand\catSh[1]{\cat{Sh}(#1)}      % category of sheaves

\newcommand\catDSh[2][]{\cat{D}^{#1}(#2)}   % derived category of sheaves
\newcommand\catDbSh[2][]{\catDSh[b]{#2}}    % bounded derived category of sheaves

\newcommand\DD{\mathbb{D}}              % Poincare-Verdier duality functor
\newcommand\dc{\omega}           % dualizing complex

\newcommand\LL{\mathbb{L}}      % left derived
\newcommand\Lotimes{\otimes^\LL}    % derived tensor product
\newcommand\RHom{\RR\Hom}           % derived Hom
\newcommand\RsheafHom{\RR\sheafHom} % derived internal Hom

\newcommand\ul\underline
\newcommand\gr{\operatorname{gr}}   % associated graded
\newcommand\from\leftarrow
\newcommand\onto\twoheadrightarrow
\newcommand\isoto{\xrightarrow{\sim}}
\newcommand\res[2]{\mathchoice{\left.#1\right|_{#2}}{#1|_{#2}}{#1|_{#2}}{#1|_{#2}}} % restriction
\newcommand\pt{\mathrm{pt}}         % point
\begin{document}

\begin{abstract}
    We introduce the notion of a \enquote{graded topological space}: a topological space endowed with a sheaf of abelian groups which we think of as a sheaf of gradings.
    Any object living on a graded topological space will be graded by this sheaf of abelian groups.
    We work out the fundamentals of sheaf theory and Poincar\'e--Verdier duality for such spaces.
\end{abstract}

\maketitle

\setcounter{tocdepth}{1}
\tableofcontents

\section{Introduction}\label{sec:intro}

Given a topological space $X$, we are interested in graded sheaves on $X$ whose grading varies over $X$.
We formalize this notion by introducing a sheaf of gradings $\Lambda$ and then considering sheaves $\sheaf F$ on $X$ such that for each open $U$ of $X$ the sections of $\sheaf F$ over $U$ are graded by $\Lambda(U)$.
In this short note we develop the basic theory of such objects.

We are mainly motivated by questions in logarithmic geometry.
However, as we expect that the constructions presented here will be useful in other situations, we chose to give this self-contained exposition of graded spaces.
Our main goal is to clarify the definition of $\Lambda$-graded objects and functors between categories of such objects, as well as to show that standard results of sheaf theory continue to hold in this generality.

\subsection*{Motivation}

Our specific motivation is the wish to classify logarithmic connections and D-modules by some analogue of the Riemann--Hilbert correspondence.
In this subsection we discuss in an example why one might expect to obtain graded sheaves in this process.
This also informs the level of generality that we have chosen for the constructions in this article.
We want to emphasize that the present subsection is not necessary for understanding the remainder of this text and may be safely skipped by the reader not interested in logarithmic connections.

Consider the complex line $X = \mathbb A^1_{\CC}$ with coordinate function $z$.
A logarithmic connection on $X$ is a vector bundle with an action of the subbundle of the tangent bundle generated by $z\frac{\partial}{\partial z}$.
So, for the trivial line bundle $\sheaf O$ we could require that $z\frac{\partial}{\partial z}$ acts on sections $f \in \sheaf O$ by $z\frac{\partial}{\partial z} \cdot f = z\frac{\partial f}{\partial z} + \lambda f$ for any fixed $\lambda \in \CC$.

On $X \setminus \{0\}$ this reduces to the usual connections $\frac{\partial}{\partial z} \cdot f = \frac{\partial f}{\partial z} + \frac{\lambda f}{z}$, which only depend on $\lambda \bmod \mathbb{Z}$.
On the other hand, as logarithmic connections the above connections are genuinely different for each $\lambda \in \mathbb{C}$.

Classically, (integrable) connections are equivalent to locally constant sheaves.
As $X = \mathbb{A}^1$ does not support any nontrivial locally constant sheaves, one modifies $X$ slightly by replacing the origin with a circle (i.e.~one takes the real blowup at the origin) \cite{KatoNakayama}.
This greatly increases the number of locally constant sheaves at our disposal.
However locally constant sheaves on this new space $X_{\log}$ can still only record $\lambda \mod \mathbb{Z}$ as the monodromy around the circle.
Thus one grades the sheaves in order to record the residue of $\lambda$ modulo $\mathbb{Z}$ \cite{Ogus, LogDR}.

In order for this construction to generalize the classical situation, one should impose this grading only over the added circle, so that over $X \setminus \{0\}$ one obtains just a classical local system.
In addition, to make this construction work when one replaces vector bundles by coherent sheaves or D-modules, one needs to not only consider sheaves of $\mathbb{C}$-modules, but modules over more general sheaves of $\mathbb{C}$-algebras (which themselves are graded), so that one can record the possible appearance of nilpotent sections.
Thus one naturally arrives at the notion of ringed graded spaces explored in Section~\ref{sec:ringed}.

Classically the Riemann--Hilbert correspondence matches the six functor formalisms of regular holonomic D-modules and (constructible) sheaves of $\CC$-modules.
Of particular importance are the duality functors that exist in both contexts.
In Section~\ref{sec:duality} we show that Poincar\'e--Verdier duality can be extended to graded sheaves, while duality for logarithmic D-modules is introduced in \cite{KT}.
In \cite{LogDR} we give an explicit computation of the dualizing functor for spaces of the form $X_{\log}$ and show that it exactly matches the duality functor for logarithmic D-modules.

\subsection*{Acknowledgments}

The author was supported by the National Science Foundation under Grant No.~DMS-1638352. 

\subsection*{Standing assumptions}

All topological spaces are assumed to be locally compact, and hence in particular Hausdorff.
By a ring we always mean a commutative ring with unit.
We write abelian groups additively with neutral element $0$.

\section{Graded topological spaces}

\begin{Definition}
    A \emph{graded topological space} is a pair $(X,\Lambda)$, consisting of a topological space $X$ and a sheaf of abelian groups $\Lambda$ on $X$.
    A morphism of graded topological spaces $(X,\Lambda_X) \to (Y,\Lambda_Y)$ consists of a pair $(f,f^\flat)$, where $f\colon X \to Y$ is a continuous map and $f^\flat\colon f^{-1}\Lambda_Y \to \Lambda_X$ is a morphism of sheaves of abelian groups.
    Such a morphism is called \emph{strict} if $f^\flat$ is an isomorphism.
\end{Definition}

We will often denote a graded topological space $(X,\Lambda)$ simply by $X$ and similarly a map $(f,f^\flat)$ by $f$.
Any topological space $X$ can be considered as a graded topological space with $\Lambda = 0$.

For an abelian group $\Lambda$, a $\Lambda$-graded $\ZZ$-module $M$ is a $\ZZ$-module with a decomposition $M = \bigoplus_{\lambda \in \Lambda} M_{\lambda}$ for $\ZZ$-modules $M_\lambda$.
If $m \in M$ is homogeneous of degree $\lambda$, we write $\deg(m) = \lambda$.

\begin{Definition}
    Let $(X, \Lambda)$ be a graded topological space.
    A \emph{presheaf} $\sheaf F$ on $(X,\Lambda)$ is an assignment of a $\Lambda(U)$-graded $\ZZ$-module $\sheaf F(U)$ to each open subset $U \subseteq X$ together with restriction maps (of $\ZZ$-modules) $\rho^U_V$ such that $\rho^U_V(\sheaf F(U)_\lambda) \subseteq \sheaf F(V)_{\res{\lambda}V}$ for each $\lambda \in \Lambda(U)$.
    Sometimes we will call such an object a \emph{$\Lambda$-graded presheaf} to emphasize the distinction with ordinary presheaves.
    
    Let $\sheaf F$, $\sheaf G$ be two presheaves on $(X,\Lambda)$.
    A \emph{morphism of $\Lambda$-graded presheaves} $\phi\colon \sheaf F \to \sheaf G$ is an ordinary morphism of presheaves such that in addition $\phi_U(\sheaf F(U)_\lambda) \subseteq \sheaf G(U)_\lambda$ for each open $U$ and $\lambda \in \Lambda(U)$.
    We write $\catPreSh{X,\Lambda}$ for the category of presheaves on $(X,\Lambda)$.

    Let $\lambda \in \Lambda(X)$ and $\sheaf F \in \catPreSh{X,\Lambda}$.
    We write $\sheaf F\langle \lambda\rangle$ for the presheaf with $\Gamma(U,\sheaf F\langle \lambda\rangle)_\mu = \Gamma(U,\sheaf F)_{\mu+\res{\lambda}U}$.
    An element of $\Hom_{\catPreSh{X,\Lambda}}(\sheaf F, \sheaf G\langle \lambda \rangle)$ is called a morphism of degree $\lambda$.
    
    There exists an obvious forgetful functor $\catPreSh{X,\Lambda} \to \catPreSh{X}$.
    We will sometimes silently treat a graded presheaf as an ordinary presheaf on $X$ via this functor.
\end{Definition}

For $x \in X$ one defines the stalk $\sheaf F_x$ of a presheaf in the usual way.
It is a $\Lambda_x$-graded $\ZZ$-module.

\begin{Definition}
    For any $\sheaf F \in \catPreSh{X,\Lambda}$ and any $\lambda \in \Lambda(X)$ we let $\sheaf F_\lambda$ be the ordinary presheaf given by
    \[
        U \mapsto \sheaf F(U)_{\res{\lambda}{U}}.
    \]
\end{Definition}

\begin{Definition}
    Let $(X,\Lambda)$ be a graded topological space.
    A \emph{($\Lambda$-graded) sheaf} on $(X,\Lambda)$ is $\Lambda$-graded presheaf $\sheaf F$ that such that for each open $U$ of $X$ and each $\lambda \in \Lambda(U)$ the (ordinary) presheaf $(\res{\sheaf F}{U})_\lambda$ is a sheaf.
    We denote by $\catSh{X,\Lambda}$ the full subcategory of $\catPreSh{X,\Lambda}$ consisting of sheaves.
\end{Definition}

\begin{Remark}
    The underlying ungraded presheaf of a graded sheaf $\sheaf F$ need not necessarily be a sheaf.
    For example, one might have two sections $s_1 \in \sheaf F(U_1)_{\lambda_1}$ and $s_2 \in \sheaf F(U_2)_{\lambda_2}$ such that $\res{s_1}{U_1 \cap U_2} = 0 = \res{s_2}{U_1 \cap U_2}$ but $\res{\lambda_1}{U_1 \cap U_2} \ne \res{\lambda_2}{U_1 \cap U_2}$.
    In this case, disregarding the grading one should be able to glue $s_1$ and $s_2$.
    But as $\lambda_1$ and $\lambda_2$ do not glue, one would not be able to assign a grading to the glued section.
\end{Remark}

As in the ungraded setting a morphism of sheaves is an isomorphism if and only if it is on stalks, see~\cite[Proposition~2.2.2]{KashiwaraSchapira:1994:SheavesOnManifolds}.
Similarly, by adding gradings to the standard construction (see~\cite[Proposition~2.2.3]{KashiwaraSchapira:1994:SheavesOnManifolds}) one defines the sheafification functor:

\begin{Lemma}
    The forgetful functor $\catSh{X,\Lambda} \to \catPreSh{X,\Lambda}$ has a left adjoint, called \emph{sheafification}.
    If $\sheaf F$ is a presheaf, then the associated sheaf has the same stalks as $\sheaf F$.
\end{Lemma}

\begin{Definition}
    For $\sheaf F, \sheaf G \in \catSh{X,\Lambda}$ we set
    \[
        \Hom^\Lambda(\sheaf F,\sheaf G) = \bigoplus_{\lambda \in \Lambda(X)} \Hom_{\catSh{X,\, \Lambda}}(\sheaf F,\, \sheaf G\langle \lambda\rangle).
    \]
    This enhances $\cat{Sh}(X,\Lambda)$ to a $\Lambda(X)$-graded category.
    We denote by $\sheafHom(\sheaf F, \sheaf G)$ the $\Lambda$-graded sheaf
    \[
        U \mapsto \Hom^{\res{\Lambda}{U}}\bigl(\res{\sheaf F}{U},\, \res{\sheaf G}{U}\bigr).
    \]
\end{Definition}

\begin{Definition}
    For $\sheaf F, \sheaf G \in \catSh{X,\Lambda}$ denote by $\sheaf F \otimes \sheaf G$ the $\Lambda$-graded sheaf associated to the presheaf
    \[
        U \mapsto \sheaf F(U) \otimes_\ZZ \sheaf G(U).
    \]
\end{Definition}

Let $f\colon X \to Y$ be a continuous map of topological spaces and let $\Lambda$ be a sheaf of abelian groups on $Y$.
Then we get an obvious morphism of graded topological spaces $f\colon (X, f^{-1}\Lambda) \to (Y,\Lambda)$.
The usual functors of sheaves $f^{-1}$ and $f_*$ induce adjoint functors between $\catSh{X, f^{-1}\Lambda}$ and $\catSh{Y,\Lambda}$.

\begin{Definition}
    Let $f\colon (X,\Lambda_X) \to (Y,\Lambda_Y)$ be a morphism of graded topological spaces.
    Define a functor
    \[
        f_{\gr}^{-1}\colon \catSh{Y,\Lambda_Y} \to \catSh{X,\Lambda_X}
    \]
    by
    \[
        \Gamma(U,f_{\gr}^{-1}\sheaf F)_\lambda = \bigl\langle s \in \Gamma(U,f^{-1}\sheaf F) : f^\flat(\deg s) = \lambda\bigr\rangle, \quad \lambda \in \Lambda_X(U).
    \]
    \label{def:graded-sheaf-pushforward}%
    Also define a functor
    \[
        f_{\gr,*}\colon \catSh{X,\Lambda_X} \to \catSh{Y,\Lambda_Y}
    \]
    by
    \[
        \Gamma(V,f_{\gr,*}\sheaf F)_\mu = \Gamma(f^{-1}V,\, \sheaf F)_{f^\flat(\mu)},\quad \mu \in \Lambda_Y(V) \to f^{-1}\Lambda_Y(f^{-1}V).
    \]
\end{Definition}

One checks that these definition indeed make sense, i.e.~send graded sheaves to graded sheaves.
We note that if $p\colon (X,\Lambda_X) \to (\pt, 0)$, then $p_{\gr,*}$ is the \enquote{degree $0$ global sections} functor.
In particular we have $p_{\gr,*}\sheafHom_{\catSh{X,\Lambda_X}}(\sheaf F, \sheaf G) = \Hom_{\catSh{X,\Lambda_X}}(\sheaf F, \sheaf G)$.

\begin{Remark}\label{rem:graded_pushforward_finitness}
    The pushforward functor will in general not keep finiteness properties of the sheaf $\sheaf F$.
    A good example to keep in mind is $Y = \RR$ with $\Lambda$ constructible such that $\res{\Lambda}{\RR^*} = 0$ and $\Lambda_0 = \ZZ$.
    Then the graded pushforward along $j\colon \RR^* \hookrightarrow \RR$ sends the constant sheaf with fiber $k$ to the sheaf with stalk at $0$ equal to $\bigoplus_{\mu \in \ZZ} k$ (and constant with fiber $k$ otherwise).
\end{Remark}

\begin{Lemma}\label{lem:sheaf-adjunction}
    Let $f\colon X \to Y$ be a morphism of graded topological spaces.
    Then for $\sheaf F \in \catSh{X,\Lambda_X}$ and $\sheaf G \in \catSh{Y,\Lambda_Y}$ there exists a natural isomorphism
    \[
        \sheafHom_{\catSh{Y,\Lambda_Y}}(\sheaf G, f_{\gr,*}\sheaf F) \cong f_{\gr,*}\sheafHom_{\catSh{X,\Lambda_X}}(f_{\gr}^{-1}\sheaf G, \sheaf F).
    \]
    In particular,
    \[
        \Hom_{\catSh{Y,\Lambda_Y}}(\sheaf G, f_{\gr,*}\sheaf F) \cong \Hom_{\catSh{X,\Lambda_X})}(f_{\gr}^{-1}\sheaf G, \sheaf F)
    \]
    and $f_{\gr}^{-1}$ is left adjoint to $f_{\gr,*}$.
\end{Lemma}

\begin{proof}
    If $\Lambda_X = f^{-1}\Lambda_Y$, then this follows easily from the classical adjointness of pullback and pushforward.
    Thus we can assume that the underlying map of topological spaces is the identity.
    In this case one checks that a morphism of ungraded sheaves is contained in either $\sheafHom$ if it fulfills the same degree conditions on local sections.
\end{proof}

The functor $f_{\gr}^{-1}$ is clearly exact, whence $f_{\gr,*}$ is left exact by Lemma~\ref{lem:sheaf-adjunction}.

\begin{Definition}
    Let $f\colon X \to Y$ be a morphism of graded spaces and $\sheaf F \in \catSh{X,\Lambda}$.
    We define $f_{\gr,!}\sheaf F$ to be the subsheaf of $f_{\gr,*}\sheaf F$ with sections
    \[
        \Gamma(U, f_{\gr,!}\sheaf F)_\mu = \bigl\{ s \in \Gamma(f^{-1}U, \sheaf F)_{f^\flat(\mu)} : f\colon \supp s \to U \text{ is proper}\bigr\}.
    \]
    We write $\Gamma_c(X,\sheaf F)$ for $p_{\gr,!}\sheaf F$ with $p\colon (X,\Lambda) \to (\pt, \Lambda(X))$.
\end{Definition}

Clearly $f_{\gr,!}$ is left exact and $f_{\gr,*} = f_{\gr,!}$ when $f$ is proper.

\begin{Remark}
    Let $f\colon X \to Y$ be a morphism of graded spaces and $\sheaf F \in \catSh{X, \Lambda_X}$.
    Then in general $\Gamma(Y,\, f_{\gr,*}\sheaf F)$ is not equal to $\Gamma(X,\, \sheaf F)$.
    For an extreme example consider $f\colon X \to (\pt,0)$, where the former is degree $0$ sections, while the latter are all ($\Lambda(X)$-graded) sections.
    The same remark applies to $\Gamma_c$ and $f_{\gr,!}$.
\end{Remark}

For any subset $i\colon Y \hookrightarrow X$ we set $\res{\sheaf F}{Y} = i_{\gr}^{-1}\sheaf F$, where we endow $Y$ with the grading $i^{-1}\Lambda$.

\begin{Lemma}\label{lem:graded_baby_basechange}
    Let $f\colon X \to Y$ be a morphism of graded spaces and let $\sheaf F \in \catSh{X,\Lambda_X}$.
    Factor $f$ as
    \[
        (X, \Lambda_X) \xrightarrow{f_1} (X,f^{-1}\Lambda_Y) \xrightarrow{f_2} (Y, \Lambda_Y).
    \]
    Then for each $y \in Y$ there exists a canonical isomorphism of $\Lambda_y$-graded modules
    \[
        (f_{\gr,!} \sheaf F)_y \isoto \Gamma_c\biggl(f^{-1}(y),\, (\res{f_1}{f^{-1}(y)})_{\gr,*}\res{\sheaf F}{f^{-1}(y)}\biggr).
    \]
    Here we endow $f^{-1}(y)$ with the sheaf of gradings $f^{-1}\Lambda_{Y,y}$.
\end{Lemma}

\begin{proof}
    One easily checks that the above morphism respects the gradings.
    The fact that it is an isomorphism can then be checked in the usual way, see~\cite[Proposition~2.5.2]{KashiwaraSchapira:1994:SheavesOnManifolds} or \cite[Theorem~VII.1.4]{Iversen:1986:CohomologyOfSheaves}.
\end{proof}

\begin{Lemma}\label{lem:graded_fiber_product}
    The category of graded spaces admits pullbacks. 
    Concretely, if $f\colon (Y_1, \Lambda_{Y_1}) \to (X, \Lambda_X)$ and $g\colon (Y_2,\Lambda_{Y_2}) \to (X,\Lambda_X)$ are two morphisms of graded spaces, then their pullback is isomorphic to $(Z,\Lambda_Z)$ as follows:
    The underlying topological space $Z$ is the cartesian product $Y_1 \times_{X} Y_2$.
    Let $\tilde f\colon Z \to Y_2$ and $\tilde g\colon Z \to Y_1$ be the projection maps.
    The sheaf of abelian groups $\Lambda_Z$ is the pushout of $\tilde g^{-1}(f^\flat)\colon \tilde g^{-1}f^{-1} \Lambda_X \to \tilde g^{-1} \Lambda_{Y_1}$ and $\tilde f^{-1}(g^\flat)\colon \tilde f^{-1}g^{-1} \Lambda_X \to \tilde f^{-1} \Lambda_{Y_2}$:
    \[
        \begin{tikzcd}
            (Y_1 \times_{X} Y_2,\, \tilde g^{-1} \Lambda_{Y_1} \oplus_{(\tilde g \circ f)^{-1} \Lambda_X} \tilde f^{-1} \Lambda_{Y_2}) \arrow[r, "\tilde{g}"] \arrow[d, "\tilde f"] & (Y_1,\,\Lambda_{Y_1}) \arrow[d,"f"] \\
            (Y_2,\, \Lambda_{Y_2}) \arrow[r,"g"] & (X,\, \Lambda_X)
        \end{tikzcd}
    \]
\end{Lemma}

\begin{proof}
    Follows directly from the universal properties.
\end{proof}

\begin{Proposition}\label{prop:underived_base_change}
    Consider a cartesian square
    \[
        \begin{tikzcd}
            Z \arrow[r, "\tilde g"] \arrow[d,"\tilde f"] & Y_1 \arrow[d, "f"] \\
            Y_2 \arrow[r, "g"] & X
        \end{tikzcd}
    \]
    of graded spaces.
    Then there is a canonical isomorphism of functors
    \[
        g_{\gr}^{-1} \circ f_{\gr,!} \isoto \tilde f_{\gr,!} \circ \tilde g_{\gr}^{-1}.
    \]
\end{Proposition}

\begin{proof}
    Using Lemma~\ref{lem:graded_baby_basechange}, this can be shown as in the ungraded situation while carefully keeping track of gradings using Lemma~\ref{lem:graded_fiber_product}, see~\cite[Proposition~2.5.11]{KashiwaraSchapira:1994:SheavesOnManifolds}.
\end{proof}

\section{Ringed graded topological spaces}\label{sec:ringed}

Let $(X,\Lambda)$ be a graded topological space.
A sheaf of rings\footnote{Recall that by \enquote{ring} we always mean a commutative ring with unit.} (resp.~$k$-algebras) on $X$ is a $\Lambda$-graded sheaf $\sheaf R$ such that each $\sheaf R(U)$ is a $\Lambda(U)$-graded ring (resp.~a $\Lambda(U)$-graded $k$-algebra) and the restriction maps are ring homomorphisms (resp.~$k$-algebra homomorphisms).

\begin{Definition}
    A \emph{graded ringed topological spaces} is a triple $(X, \Lambda_X, \sheaf R_X)$, where $(X,\Lambda_X)$ is a graded topological space and $\sheaf R_X$ is a $\Lambda_X$-graded sheaf of commutative rings on $X$.
    A morphism of graded ringed topological spaces $(X,\Lambda_X,\sheaf R_X) \to (Y,\Lambda_Y,\sheaf R_Y)$ is a triple $(f,f^\flat,f^\sharp)$ where $(f,f^\flat)$ is a morphism of graded topological spaces and $f^\sharp\colon f_{\gr}^{-1}\sheaf R_Y \to \sheaf R_X$ is morphism of $\Lambda_X$-graded sheaves of rings.
    Such a morphism is called \emph{strict} if $f^\flat$ and $f^\sharp$ are isomorphisms.
\end{Definition}

\begin{Definition}
    Let $(X,\Lambda,\sheaf R)$ be a graded ringed topological space.
    We write $\catSh{X,\Lambda,\sheaf R}$ for the category of $\Lambda$-graded sheaves of $\sheaf R$-modules, i.e.~the category whose objects are $\Lambda$-graded sheaves $\sheaf F$ such that each $\sheaf F(U)$ is a $\Lambda(U)$-graded $\sheaf R(U)$-module with compatible restriction maps and morphisms are required to respect this additional structure.
\end{Definition}

Let $\sheaf F$ and $\sheaf G$ be two $\sheaf R$-modules.
Then $\Hom^\Lambda_{\catSh{X,\Lambda,\sheaf R}}(\sheaf F,\sheaf G)$, $\sheaf F \otimes_{\sheaf R} \sheaf G$, and $\sheafHom_{\catSh{X,\Lambda,\sheaf R}}(\sheaf F,\sheaf G)$ are defined in the obvious way.
We will often simply write $\Hom_{\sheaf R}$, $\Hom_{\sheaf R}^\Lambda$ and $\sheafHom_{\sheaf R}$ for the various Hom functors.

\begin{Lemma}\label{lem:module-tensor-hom-adjunction}
    Let $\sheaf R \to \sheaf S$ be a morphism of $\Lambda$-graded sheaves of commutative rings.
    Let $\sheaf F$ and $\sheaf H$ be $\sheaf S$-modules and $\sheaf G$ an $\sheaf R$-module.
    Then there is a canonical isomorphism
    \[
        \sheafHom_{\catSh{X,\Lambda,\sheaf S}}(\sheaf F \otimes_{\sheaf R} \sheaf G,\, \sheaf H)
        \cong
        \sheafHom_{\catSh{X,\Lambda,\sheaf R}}(\sheaf G,\, \sheafHom_{\catSh{X,\Lambda,\sheaf S}}(\sheaf F, \sheaf H)).
    \]
\end{Lemma}

\begin{proof}
    As in the ungraded setting, it suffices to check the isomorphism on presheaves defining the above sheaves, see~\cite[Proposition~2.2.9]{KashiwaraSchapira:1994:SheavesOnManifolds}.
    There it follows from the corresponding adjunction for graded modules.
\end{proof}

\begin{Lemma}\label{lem:inverse_and_tensor}
    Let $f\colon (X,\Lambda_X) \to (Y,\Lambda_Y)$ be a morphism of graded topological spaces and let $\sheaf R$ be a $\Lambda_Y$-graded sheaf of rings on $Y$.
    Then for any $\sheaf R$-modules $\sheaf F$ and $\sheaf G$ there exists a canonical isomorphism
    \[
        f_{\gr}^{-1}\sheaf F \otimes_{f_{\gr}^{-1}\sheaf R} f_{\gr}^{-1}\sheaf G \cong f_{\gr}^{-1}(\sheaf F \otimes_{\sheaf R} \sheaf G).
    \]
\end{Lemma}

\begin{proof}
    As in the ungraded setting, see~\cite[Proposition~2.3.5]{KashiwaraSchapira:1994:SheavesOnManifolds}.
\end{proof}

Clearly, $\catSh{X,\Lambda} = \catSh{X,\Lambda,\ZZ}$.
If $f\colon X \to Y$ is a morphism of graded ringed topological spaces, then $f_{\gr,*}$ as defined in Definition~\ref{def:graded-sheaf-pushforward} enhances to a functor
\[
    f_{\gr,*}\colon \catSh{X,\Lambda_X,\sheaf R_X} \to \catSh{Y,\Lambda_Y,\sheaf R_Y},
\]
and similarly we have a functor
\[
    f_{\gr,!}\colon \catSh{X,\Lambda_X,\sheaf R_X} \to \catSh{Y,\Lambda_Y,\sheaf R_Y}.
\]

\begin{Remark}
    Here we have to be careful to make sure that $\sheaf R_Y$ acts with the correct degrees: 
    If $t\in\sheaf R_Y(V)_\lambda$ and $m \in f_{\gr,*}\sheaf F(V)_\mu$, then $m$ comes from a section in $\sheaf F(f^{-1}V)_{f^\flat(\mu)}$.
    Via the morphism $f_{\gr}^{-1}\sheaf R_Y \to \sheaf R_X$, $rm$ comes from a section $rm \in \sheaf F(f^{-1}V)_{f^\flat(\mu) + f^\flat(\lambda)}$.
    A priory there might be many $\mu' \in \Lambda_Y(V)$ which map to $f^\flat(\mu) + f^\flat(\lambda)$.
    The section $rm \in f_{\gr,*}\sheaf F(V)$ has to be in degree $\mu + \lambda$.

    Again a good example to keep in mind is as in Remark~\ref{rem:graded_pushforward_finitness}, where one endows $\RR$ with the constructible sheaf of rings with stalks $\CC$ on $\RR^*$ and $\CC[t]$ at $0$ with $\deg t = 1$ and $\res{t}{\RR^*}=1$.
    Let $\sheaf F$ be the constant sheaf with stalk $k$ on $\RR^*$ and $j$ the inclusion $\RR^*\hookrightarrow \RR$.
    If $m \in j_{\gr,*}\sheaf F(V)_0$ (with $0 \in V$ open), then $tm$ comes from a section $\res{t}{\RR^*}\res{m}{\RR^*} = \res{m}{\RR^*} \in \sheaf F(V\setminus 0)_0$.
    The sheaf $j_{\gr,*}\sheaf F$ contains $\ZZ$-many copies of this section.
    We have to define $tm$ to be the one in $f_{\gr,*}\sheaf F(V)_1$.
\end{Remark}

\begin{Definition}
    Let $f\colon X \to Y$ be a morphism of graded ringed topological spaces.
    Define a functor
    \[
        f_{\gr}^*\colon \catSh{Y,\Lambda_Y,\sheaf R_Y} \to \catSh{Y,\Lambda_X,\sheaf R_X},
        \qquad
        \sheaf F \mapsto f_{\gr}^{-1}\sheaf F \otimes_{f_{\gr}^{-1}\sheaf R_Y} \sheaf R_X.
    \]
\end{Definition}

\begin{Lemma}\label{lem:module-adjunction}
    Let $f\colon X \to Y$ be a morphism of graded ringed topological spaces.
    Then for $\sheaf F \in \catSh{X,\Lambda_X,\sheaf R_X}$ and $\sheaf G \in \catSh{Y,\Lambda_Y,\sheaf R_Y}$ there exists a natural isomorphism
    \[
        \sheafHom_{\sheaf R_Y}(\sheaf G,\,  f_{\gr,*}\sheaf F) \cong f_{\gr,*}\sheafHom_{\sheaf R_X}(f_{\gr}^{*}\sheaf G,\, \sheaf F).
    \]
    In particular,
    \[
        \Hom_{\sheaf R_Y}(\sheaf G,\,  f_{\gr,*}\sheaf F) \cong \Hom_{\sheaf R_X}(f_{\gr}^{*}\sheaf G,\, \sheaf F)
    \]
    and $f_{\gr}^{*}$ is left adjoint to $f_{\gr,*}$.
\end{Lemma}

\begin{proof}
    If $\sheaf R_X = f_{\gr}^{-1}\sheaf R_Y$, then the statement is proven in the same way as Lemma~\ref{lem:sheaf-adjunction}.
    So we can assume that $(X,\Lambda_X) = (Y,\Lambda_Y)$.
In this case the statement is just tensor-Hom adjunction (Lemma~\ref{lem:module-tensor-hom-adjunction}).
\end{proof}

As $f_{\gr,*}$ is left exact, Lemma~\ref{lem:module-adjunction} implies that $f_{\gr}^*$ is right exact.

\begin{Definition}
    For an $\sheaf R$-module $\sheaf F$ and a locally closed subset $Y \subseteq X$ we write $\sheaf F_Y$ for the sheaf satisfying the following conditions
    \[
        \res{\sheaf F_Y}{Y} = \res{\sheaf F}{Y} \quad\text{and}\quad \res{\sheaf F_Y}{X\setminus Y} = 0.
    \]
\end{Definition}

The sheaf $\sheaf F_Y$ is constructed in the usual way, see~\cite[93]{KashiwaraSchapira:1994:SheavesOnManifolds}.
If $\sheaf F$ is a $\Lambda$-graded  $\sheaf R$-module, then so is $\sheaf F_Y$.
The following lemma is standard.

\begin{Lemma}\label{lem:module-basic-triangle}
    Let $Y \subseteq X$ be a locally closed subset.
    The functor $\sheaf F \mapsto \sheaf F_Y$ is exact.
    Further, if $U\subseteq X$ is open, then we have an exact sequence in $\catSh{X,\Lambda,\sheaf R}$
    \[
        0 \to \sheaf F_U \to \sheaf F \to \sheaf F_{X\setminus U} \to 0.
        \qedhere
    \]
\end{Lemma}

\section{Derived categories}

In this section $(X,\Lambda,\sheaf R)$ will always be a ringed graded space.

As in the non-graded case one defines the kernel and cokernel of a morphism of $\sheaf R$-modules and obtains the following lemma (see~\cite[Proposition~2.2.4]{KashiwaraSchapira:1994:SheavesOnManifolds}).

\begin{Lemma}
    The category $\catSh{X,\Lambda,\sheaf R}$ is abelian.
\end{Lemma}

We write $\catDSh[*]{X,\Lambda,\sheaf R}$ for the corresponding derived categories, where $*$ is one of $\emptyset,\, {+},\, {-},\, b$.

\begin{Lemma}\label{lem:enough_flats}
    Every $\sheaf R$-module $\sheaf F \in \catSh{X,\Lambda,\sheaf R}$ admits a surjection $\sheaf P \to \sheaf F$ for some flat $\sheaf R$-module $\sheaf P$.
\end{Lemma}

\begin{proof}
    For each open $U \subseteq X$ and each homogeneous section $s \in \sheaf F(U)_\lambda$ set $\sheaf P(U,s) = \sheaf R_U\langle-\lambda\rangle$.
    Then $\sheaf P(U,s)$ has a map to $\sheaf F$ sending $1$ to $s$.
    Thus $\sheaf P = \bigoplus_{U,s} \sheaf P(U,s)$ maps onto $\sheaf F$.
    Further $\sheaf P$ is flat since for each $x \in X$ the stalk $\sheaf P_x$ is a sum of shifts of free $\sheaf R_x$-modules.
\end{proof}

\subsection{The derived category via model structures}

Let $\cat{Ch}(X,\Lambda,\sheaf R)$ be the category of complexes of $\sheaf R$-modules.

\begin{Proposition}\label{prop:monoidal-model-structure}
    The category $\cat{Ch}(X,\Lambda,\sheaf R)$ can be endowed with a symmetric monoidal model structure such that the weak equivalences are the quasi-equivalences of complexes and the monoidal product is given by the tensor product of complexes.
    In particular $\bigl(\catDSh{X,\Lambda,\sheaf R}, \Lotimes_{\sheaf R}, \sheaf R, \RsheafHom_{\sheaf R}\bigr)$ is a closed monoidal category.
\end{Proposition}

\begin{proof}
    The proof of this proposition is along the lines of that for \cite[Proposition~2.18]{DellAmbrogioStevenson:2013:DerivedCategoryOfAGradedCommutativeNoetherianRing}.
    Thus we let $\mathcal{G}$ be the set of sheaves $\sheaf R_U\langle \lambda \rangle$, where $U$ runs over all open subsets of $X$ and $\lambda \in \Lambda(U)$.
    Then $\mathcal{G}$ is a flat family of generators in the sense of \cite[Section~3.1]{CisinskiDeglise:2009:LocalAndStableHomologicalAlgebraInGrothendieckCategories}.
    By \cite[Remark~2.12]{CisinskiDeglise:2009:LocalAndStableHomologicalAlgebraInGrothendieckCategories} we can complete $\mathcal G$ to a descent structure $(\mathcal G,\mathcal H)$, which is automatically flat by \cite[Proposition~3.7]{CisinskiDeglise:2009:LocalAndStableHomologicalAlgebraInGrothendieckCategories}.
    Thus the corresponding $\mathcal{G}$-model structure on $\cat{Ch}(X,\Lambda,\sheaf R)$ yields a symmetric monoidal model category \cite[Proposition~3.2]{CisinskiDeglise:2009:LocalAndStableHomologicalAlgebraInGrothendieckCategories}.
    %In particular the tensor product is a left Quillen functor with adjoint given by $\sheafHom_{\sheaf R}$.
    The theorem then follows from \cite[Theorem~4.3.2]{Hovey:1999:ModelCategories}.
\end{proof}

\subsection{Acyclic sheaves}

In this section we introduce several properties of sheaves and show that they imply acyclicity for various functors.

First, as usual one calls a sheaf $\sheaf I \in \catSh{X,\Lambda,\sheaf R}$ \emph{injective} if $\Hom_{\sheaf R}({-},\,\sheaf I)$ is an exact functor.

\begin{Lemma}
    The category $\catSh{X,\Lambda,\sheaf R}$ has enough injectives.
\end{Lemma}

\begin{proof}
%    Let $\hat X$ be the space $X$ endowed with the discrete topology and let $f\colon \hat X \to X$ be the natural map.
%    Endow $\hat X$ with the graded ringed structure given by $f^{-1}\Lambda$ and $f^{-1}\sheaf R$.
%    Let $\sheaf F \in \cat{Sh}(X,\Lambda,\sheaf R)$ and suppose we find a monomoprhism $f_{\gr}^{-1}\sheaf F \to \sheaf I$ with $\sheaf I$ injective on $(\hat X,f^{-1}\Lambda,f^{-1}\sheaf R)$.
%    As $f_{\gr}^{-1}$ is exact, the adjunction of Lemma~\ref{lem:sheaf-adjunction} implies that $f_{\gr,*}\sheaf I$ is injective in $\cat{Sh}(X,\Lambda,\sheaf R)$ and as $f_{*,gr}$ is left exact we obtain a monomorphism $\sheaf F \to f_{\gr,*}\sheaf I$.
%    Hence it suffices to show that $\hat X$ has enough injectives.
%
%    Let us note that since $\hat X$ is discrete, we have $\Hom_{f^{-}\sheaf R}(\sheaf G, \sheaf H) = \prod_{x \in X} \Hom_{\sheaf R_x}(\sheaf G_x,\sheaf H_x)$.
%    Fix a graded sheaf $\sheaf F$ on $\hat X$.
%    For each $x \in \hat X$ let $I_x$ be an injective $\Lambda_x$-graded $\sheaf R_x$-module with a monomorphism $\sheaf F_x \to I_x$, which exists by \stackcite{04JD}.
%    Then $\sheaf I = \prod_{x \in X} I_x$ is injective with a monomorphism $\sheaf F \to \sheaf I$.
    As in the ungraded situation one reduces to the case of $X$ being a single point, see~\cite[Proposition~2.4.3]{KashiwaraSchapira:1994:SheavesOnManifolds}.
    There the statement follows from the corresponding statement for graded modules, which is classical (see for example \stackcite{04JD}).
\end{proof}

\begin{Lemma}
    Let $\sheaf I$ be an injective object of $\catSh{X,\Lambda,\sheaf R}$.
    Then $\sheaf I\langle \lambda\rangle$ is injective for all $\lambda \in \Lambda(X)$ and $\res{\sheaf I}{U}$ is injective in $\catSh{U,\res{\Lambda}{U},\res{\sheaf R}{U}}$ for all open subsets $U \subseteq X$.
    In particular $\Hom_{\sheaf R}^\Lambda({-},\,\sheaf I)$ and $\sheafHom_{\sheaf R}({-},\,\sheaf I)$ are exact functors.
\end{Lemma}

\begin{proof}
    The first statement follows from $\Hom_{\sheaf R}(\sheaf F,\, \sheaf I\langle \lambda\rangle) = \Hom_{\sheaf R}(\sheaf F\langle -\lambda\rangle,\, \sheaf I)$.
    If $\sheaf G$ is an $\res{R}{U}$-module, then
    \[
        \Hom_{\res{\sheaf R}{U}}(\sheaf G,\, \res{\sheaf I}{U}) = \Hom_{\sheaf R}\bigl( (j_{\gr,*}\sheaf G)_U,\, \sheaf I\bigr),
    \]
    where $j\colon U \hookrightarrow X$ is the inclusion.
    As $j_{\gr,*}$ and ${-}_U$ are exact, the second statement follows.
\end{proof}

\begin{Definition}
    A sheaf $\sheaf F \in \catSh{X,\Lambda}$ is called \emph{flabby} if for any open subset $U \subseteq X$ and any $\lambda \in \Lambda(U)$ the sheaf $(\res{\sheaf F}{U})_\lambda$ is flabby as an ordinary sheaf.
    In other words, for any open $V \subseteq U \subseteq X$ we require that the restriction morphism $\sheaf F(U)_\lambda \to \sheaf F(V)_{\res{\lambda}{V}}$ is surjective.
\end{Definition}

Unless $\Lambda$ is flabby, a flabby graded sheaf will not necessarily be flabby as an ordinary (pre-)sheaf.

\begin{Lemma}
    Let $\sheaf I$ be injective.
    Then for every $\sheaf F \in \catSh{X,\Lambda,\sheaf R}$ the sheaf $\sheafHom_{\sheaf R}(\sheaf F,\, \sheaf I)$ is flabby.
    In particular every injective $\sheaf R$-module is flabby.
\end{Lemma}

\begin{proof}
    Let $U \subseteq X$ be open.
    Consider the short exact sequence
    \[
        0 \to \sheaf F_U \to \sheaf F \to \sheaf F_{X\setminus U} \to 0.
    \]
    Applying the exact functor $\Hom^\Lambda_{\sheaf R}({-},\,\sheaf I)$ we get a surjection
    \[
        \Gamma(X,\, \sheafHom_{\sheaf R}(\sheaf F,\,\sheaf I)) = \Hom^\Lambda_{\sheaf R}(\sheaf F,\, \sheaf I)
        \onto
        \Hom^\Lambda_{\sheaf R}(\sheaf F_U,\,\sheaf I).
    \]
    We now have
    \begin{align*}
        \Hom^\Lambda_{\sheaf R}(\sheaf F_U,\,\sheaf I) & =
        \bigoplus_{\lambda \in \Lambda(X)} \Hom_{\sheaf R}(\sheaf F_U,\,\sheaf I\langle \lambda \rangle) \\ &=
        \bigoplus_{\lambda \in \Lambda(X)} \Hom_{\res{\sheaf R}U}(\res{\sheaf F}{U},\, \res{\sheaf I}{U}\langle \res{\lambda}{U} \rangle) \\ &=
        \bigoplus_{\lambda \in \Lambda(U)} \bigoplus_{\mu \in (\rho^X_U)^{-1}(\lambda)} \Hom_{\res{\sheaf R}U}(\res{\sheaf F}{U},\, \res{\sheaf I}{U}\langle \lambda \rangle).
    \end{align*}
    The statement follows.
\end{proof}

\begin{Lemma}
    Let $0 \to \sheaf F \to \sheaf G \to \sheaf H \to 0$ be an exact sequence in $\catSh{X,\Lambda_X}$ with $\sheaf F$ flabby, and let $f\colon (X,\Lambda_X) \to (Y,\Lambda_Y)$ be a morphism of graded spaces.
    Then $0 \to f_{\gr,*}\sheaf F \to f_{\gr,*}\sheaf G \to f_{\gr,*}\sheaf H \to 0$ is exact.
    In particular $0 \to \Gamma(X,\sheaf F) \to \Gamma(X,\sheaf G) \to \Gamma(X,\sheaf H) \to 0$ is a short exact sequence of $\Lambda(X)$-graded $\sheaf R(X)$-modules.
\end{Lemma}

\begin{proof}
    It suffices to show that for every $\lambda \in f^\flat(\Lambda_Y)$ the sequence
    \[
        0 \to f_*(\sheaf F_\lambda) \to f_*(\sheaf G_\lambda) \to f_*(\sheaf H_\lambda) \to 0
    \]
    is exact.
    By definition, $\sheaf F_\lambda$ is flabby as an ordinary sheaf, so this assertion is classical, see~\cite[Proposition~2.4.7]{KashiwaraSchapira:1994:SheavesOnManifolds}.
\end{proof}

\begin{Definition}
    A sheaf $\sheaf F \in \catSh{X,\Lambda}$ is called \emph{soft} if for any open subset $U \subseteq X$ and $\lambda \in \Lambda(U)$ the sheaf $(\res{\sheaf F}{U})_\lambda$ is soft as an ordinary sheaf, i.e.~for every compact subset $i\colon K \hookrightarrow U$ the restriction $\Gamma\bigl(U,(\res{\sheaf F}{U})_\lambda\bigr) \to \Gamma\bigl(K,i^{-1}(\res{\sheaf F}{U})_\lambda\bigr)$ is surjective.
\end{Definition}

Every flabby sheaf (and hence every injective sheaf) is soft.

\begin{Lemma}\label{lem:soft-is-!-acyclic}
    Let $0 \to \sheaf F \to \sheaf G \to \sheaf H \to 0$ be an exact sequence in $\catSh{X,\Lambda_X}$ with $\sheaf F$ soft, and let $f\colon (X,\Lambda_X) \to (Y,\Lambda_Y)$ be a morphism of graded spaces.
    Then $0 \to f_{\gr,!}\sheaf F \to f_{\gr,!}\sheaf G \to f_{\gr,!}\sheaf H \to 0$ is exact.
\end{Lemma}

\begin{proof}
    It suffices to show that for every $\lambda \in f^\flat(\Lambda_Y)$ the sequence
    \[
        0 \to f_!(\sheaf F_\lambda) \to f_!(\sheaf G_\lambda) \to f_!(\sheaf H_\lambda) \to 0
    \]
    is exact.
    By definition, $\sheaf F_\lambda$ is soft as an ordinary sheaf, so this assertion is classical, see \cite[Proposition~2.5.8]{KashiwaraSchapira:1994:SheavesOnManifolds}.
\end{proof}

\subsection{Identities for derived functors}

As in the ungraded setting, one sees that if $f\colon X \to Y$ is a morphism of graded topological spaces and $\sheaf F \in \catSh{X,\Lambda_X}$ is flabby (resp.~soft), then so is $f_{\gr,*}\sheaf F$ (resp.~$f_{\gr,!}\sheaf F$).
If $f\colon X \to Y$ and $g \colon Y \to Z$ are two morphisms of graded topological spaces, then $\RR f_{\gr,*}$ (resp.~$\RR f_{\gr,!}$) can be computed via flabby (soft) resolutions.
Thus $\RR(g \circ f)_{\gr,*} \cong \RR g_{\gr,*} \circ \RR f_{\gr,*}$ and $\RR(g \circ f)_{\gr,!} \cong \RR g_{\gr,!} \circ \RR f_{\gr,!}$.
From the following lemma it then follows that also $\LL(g \circ f)_{\gr}^* \cong \LL f_{\gr}^* \circ \LL g_{\gr}^*$.

\begin{Lemma}
    Let $f\colon X \to Y$ be a morphism of graded ringed topological spaces.
    Then for $\sheaf F \in \catDSh[+]{X,\Lambda_X,\sheaf R_X}$ and $\sheaf G \in \catDSh[-]{Y,\Lambda_Y,\sheaf R_Y}$ there exists natural isomorphisms
    \begin{align*}
        \RsheafHom_{\sheaf R_Y}(\sheaf G, \RR f_{\gr,*}\sheaf F) &\cong \RR f_{\gr,*}\sheafHom_{\sheaf R_X}(\LL f_{\gr}^{*}\sheaf G, \sheaf F),
        \intertext{and}
        \RHom_{\sheaf R_Y}(\sheaf G, \RR f_{\gr,*}\sheaf F) & \cong \RHom_{\sheaf R_X}(\LL f_{\gr}^{*}\sheaf G, \sheaf F).
        \intertext{In particular,}
        \Hom_{\catDSh{Y,\Lambda_Y,\sheaf R_Y}}(\sheaf G, \RR f_{\gr,*}\sheaf F) &\cong \Hom_{\catDSh{X,\Lambda_X,\sheaf R_X}}(\LL f_{\gr}^{*}\sheaf G, \sheaf F)
    \end{align*}
    and $\LL f_{\gr}^{*}$ is left adjoint to $\RR f_{\gr,*}$.
\end{Lemma}

\begin{proof}
    By tensor-hom adjunction (Proposition~\ref{prop:monoidal-model-structure}) we can reduce to $\sheaf R_X = f^{-1}_{\gr}\sheaf R_Y$.
    By adjunction the functor $f_{\gr,*}\colon \catSh{X,\Lambda_X,f_{\gr}^{-1}\sheaf R_Y} \to \catSh{Y,\Lambda_Y,\sheaf R_{\Lambda_Y}}$ sends injective modules to injective modules.
    Thus both sides are computed via the same derived functor.
\end{proof}

\begin{Lemma}\label{lem:basic_triangle}
    Let $j\colon (U,\res{\Lambda}U) \hookrightarrow (X,\Lambda)$ be an open subset with complement $i\colon (Z, \res{\Lambda}{Z}) \hookrightarrow (X,\Lambda)$.
    Then for any $\sheaf F \in \catDSh[+]{X,\Lambda}$ there exits a distinguished triangle
    \[
        \RR j_{\gr,!}j_{\gr}^{-1}\sheaf F \to \sheaf F \to \RR i_{\gr,!}i_{\gr}^{-1}\sheaf F.
    \]
\end{Lemma}

\begin{proof}
    Since the restrictions preserve softness, it suffices to show that for any soft sheaf $\sheaf F$ we have a short exact sequence
    \[
        0 \to j_{\gr,!}j_{\gr}^{-1}\sheaf F \to \sheaf F \to i_{\gr,!}j_{\gr}^{-1}\sheaf F \to 0.
    \]
    This follows from Lemma~\ref{lem:module-basic-triangle}.
\end{proof}

\begin{Proposition}[Projection formula]\label{prop:projection_formula}
    Let $f\colon (X,\Lambda_X,\sheaf R_X) \to (Y, \Lambda_Y, \sheaf R_Y)$ be a morphism of graded ringed spaces.
    Let $\sheaf F \in \catDSh[+]{X,\Lambda_X,\sheaf R_X}$ and $\sheaf G \in \catDSh[+]{Y,\Lambda_Y,\sheaf R_Y}$ and assume that $\sheaf G$ has a finite flat resolution.
    Then there exists a canonical isomorphism
    \[
        (\RR f_{\gr,!} \sheaf F) \Lotimes_{\sheaf R_Y} \sheaf G \isoto \RR f_{\gr,!}(\sheaf F \Lotimes_{\sheaf R_X} \LL f_{\gr}^* \sheaf G).
    \]
\end{Proposition}

\begin{proof}
    Let us first assume that $\sheaf R_X = f_{\gr}^{-1}\sheaf R_Y$.

    Assume further that $\sheaf G$ is a flat $\sheaf R_Y$-module and $\Lambda_X = f^{-1}\Lambda_Y$.
    Then one shows that $f_{\gr,!}\sheaf F \otimes_{\sheaf R_Y} \sheaf G \isoto f_{\gr,!}(\sheaf F \otimes_{\sheaf R_X} f_{\gr}^{-1} \sheaf G)$ with a direct adaptation of the proof in the ungraded case, see~\cite[\nopp VII.2.4]{Iversen:1986:CohomologyOfSheaves} or \cite[Proposition~2.5.13]{KashiwaraSchapira:1994:SheavesOnManifolds}.
    On the other hand, if $X = Y$, then it is a simple matter to check that the gradings on the two sides match.

    Further, still assuming that $\sheaf G$ is flat, \cite[Lemma~2.5.12]{KashiwaraSchapira:1994:SheavesOnManifolds} (whose proof again upgrades to the graded setting) implies that both derived functors are computed by a soft resolution of $\sheaf F$, and hence agree.
    Resolving a general $\sheaf G$ by flat sheaves, the derived statement follows in the case that $\sheaf R_X = f_{\gr}^{-1}\sheaf R_Y$.

    The general case then follows from
    \begin{align*}
        \RR f_{\gr,!}(\sheaf F \Lotimes_{\sheaf R_X} \LL f_{\gr}^* \sheaf G) & =
        \RR f_{\gr,!}(\sheaf F \Lotimes_{\sheaf R_X} \sheaf R_X \Lotimes_{f_{\gr}^{-1} \sheaf R_Y} f_{\gr}^{-1} \sheaf G) \\ &=
        \RR f_{\gr,!}(\sheaf F  \Lotimes_{f_{\gr}^{-1} \sheaf R_Y} f_{\gr}^{-1} \sheaf G).
        \qedhere
    \end{align*}
\end{proof}

Similarly one upgrades Proposition~\ref{prop:underived_base_change} to a derived statement (compare \cite[Proposition~2.6.7]{KashiwaraSchapira:1994:SheavesOnManifolds}):

\begin{Proposition}
    Consider a cartesian square
    \[
        \begin{tikzcd}
            Z \arrow[r, "\tilde g"] \arrow[d,"\tilde f"] & Y_1 \arrow[d, "f"] \\
            Y_2 \arrow[r, "g"] & X
        \end{tikzcd}
    \]
    of graded spaces.
    Then there exists a canonical isomorphism
    \[
        g_{\gr}^{-1} \circ \RR f_{\gr,!} \isoto \tilde \RR f_{\gr,!} \circ \tilde g_{\gr}^{-1}.
    \]
\end{Proposition}

\begin{Remark}
    This base change isomorphism does not upgraded to a base change isomorphism for \emph{ringed} graded spaces.
    This is simply because base change doesn't even hold for general morphisms of ungraded ringed spaces (e.g.~complex (analytic) varieties).
\end{Remark}

\section{Poincar\'e--Verdier duality}\label{sec:duality}

Throughout this section we will assume that all rings are noetherian.

Recall that a topological space $X$ has cohomological dimension at most $n$ if $H^k(\Gamma_c(X,\sheaf F)) = 0$ for all $\sheaf F \in \catSh{X}$ and all $k \ge n$.

\begin{Definition}
    A graded topological space $(X,\Lambda)$ has cohomological dimension at most $n$ if the underlying topological space has cohomological dimension at most $n$.
\end{Definition}

\begin{Lemma}\label{lem:coh_dim_and_soft_sequence}
    A graded space $(X,\Lambda)$ has cohomological dimension at most $n$ if and only if for any exact sequence
    \[
        0 \to \sheaf F_0 \to \sheaf F_1 \to \dotsb \to \sheaf F_{n+1} \to 0
    \]
    in $\catSh{X,\, \Lambda}$, if $\sheaf F_1, \dotsc, \sheaf F_n$ are soft then so is $\sheaf F_{n+1}$.
\end{Lemma}

\begin{proof}
    The statement is classical if $\Lambda = 0$ \cite[Proposition~III.9.9]{Iversen:1986:CohomologyOfSheaves}.
    The graded statement follows from this by considering the sequences $\sheaf F_{\bullet,\lambda}$ for $\lambda \in \Lambda$.
\end{proof}

Recall our standing assumption that all topological spaces we consider are locally compact.
The main result of this section is the following duality theorem.

\begin{Theorem}\label{thm:duality}
    Let $f\colon X \to Y$ be a morphism of graded ringed spaces.
    Assume that $X$ has finite cohomological dimension.
    Then there exists a functor $f_{\gr}^!\colon \catDSh[+]{Y,\Lambda_Y,\sheaf R_Y} \to \catDSh[+]{X,\Lambda_X,\sheaf R_X}$ right adjoint to $\RR f_{\gr,!}$.
    Moreover there exists natural isomorphisms
    \[
        \RHom_{\sheaf R_Y}(\RR f_{\gr,!} \sheaf F, \sheaf G) \cong \RHom_{\sheaf R_X}(\sheaf F, f_{\gr}^!\sheaf G)
    \]
    and
    \[
        \RsheafHom_{\sheaf R_Y}(\RR f_{\gr,!} \sheaf F, \sheaf G) \cong f_{\gr,*}\RsheafHom_{\sheaf R_X}(\sheaf F, f_{\gr}^!\sheaf G)
    \]
\end{Theorem}

The proof of Theorem~\ref{thm:duality} is roughly the same as in the ungraded setting.
We will highlight the major steps.

\begin{Lemma}\label{lem:representability_criterion}
    Let $F\colon \catSh{X,\Lambda,\sheaf R} \to \opcat{\catMod{\sheaf R(X)}}$ be an additive functor that sends colimits to limits.
    Then $F$ is representable.
\end{Lemma}

\begin{proof}
    We define a presheaf $\sheaf F$ of $\Lambda$-graded $\sheaf R$-modules by $\sheaf F(U)_\lambda = F(\sheaf R_U\langle - \lambda \rangle)$ for each open subset $U \subseteq X$ and $\lambda \in \Lambda(U)$.
    Then $\sheaf F$ is a sheaf.
    Indeed, if $\{U_\alpha\}$ is an open covering of an open subset $U$ of $X$ and $\lambda \in \Lambda(U)$ we have an exact sequence
    \[
        \bigoplus_{\alpha,\beta} \sheaf R_{U_\alpha \cap U_\beta}\langle\res{\lambda}{U_\alpha \cap U_\beta}\rangle
        \to
        \bigoplus_{\alpha} \sheaf R_{U_\alpha}\langle\res{\lambda}{U_\alpha}\rangle
        \to
        \sheaf R_U\langle\lambda\rangle
        \to
        0.
    \]
    Applying $F$, we obtain
    \[
        0 \to
        F(\sheaf R_U\langle\lambda\rangle)
        \to
        \prod_{\alpha} F\bigl(\sheaf R_{U_\alpha}\langle\res{\lambda}{U_\alpha}\rangle\bigr)
        \to
        \prod_{\alpha,\beta} F\bigl(\sheaf R_{U_\alpha \cap U_\beta}\langle\res{\lambda}{U_\alpha \cap U_\beta}\rangle\bigr),
    \]
    which is just the sheaf condition for $\sheaf F$.
    
    We can write any sheaf $\sheaf G \in \catSh{X,\Lambda,\sheaf R}$ functorially as a colimit of sheaves of the form $\sheaf R_U\langle \lambda \rangle$.
    Namely, we form the category whose objects are pairs $(U,s)$ with $U \subseteq X$ open and $s$ a homogeneous element of $\sheaf G(U)$ and with a single morphism $(U,s) \to (U',s')$ if and only if $U \subseteq U'$ and $s = \res{s'}{U}$.
    For each such pair we have a map $\sheaf R_U\langle-\deg s\rangle \to \sheaf G$ defined by the section $s$.

    It follows from the assumption on $F$ that we have a natural isomorphism $\Hom(\sheaf G, \sheaf F) \cong F(\sheaf G)$.
\end{proof}

\begin{Lemma}\label{lem:duality_representability}
    Let $f\colon X \to Y$ be a morphism of ringed graded topological spaces and assume that $X$ has finite cohomological dimension.
    Then for any flat and soft $\sheaf R_X$-module $\sheaf M$ on $X$ and any $\sheaf R_Y$-module $\sheaf G$ the functor 
    \[
        \sheaf F \mapsto \Hom_{\sheaf R_Y}(f_{\gr,!}(\sheaf F \otimes_{\sheaf R_X} \sheaf M), \sheaf G)
    \]
    is representable.
\end{Lemma}

\begin{proof}
    By Lemma~\ref{lem:representability_criterion}, it suffices to show that the functor $\sheaf F \mapsto f_{\gr,!}(\sheaf F \otimes_{\sheaf R_X} \sheaf M)$ commutes with colimits.
    As in the ungraded case the functor commutes with direct sums, so it suffices to show that it is exact.
    For this it in turn suffices to show that $\sheaf F \otimes_{\sheaf R_X} \sheaf M$ is soft.

    By the construction of Lemma~\ref{lem:enough_flats}, $\sheaf F$ has a resolution
    \[
        \cdots \to \sheaf F^{-2} \to \sheaf F^{-1} \to \sheaf F^{0} \to \sheaf F \to 0
    \]
    such that each $\sheaf F^i$ is a direct product of sheaves of the form $\sheaf R_U\langle \lambda \rangle$ for $U \subseteq X$ open and $\lambda \in \Lambda_X(U)$.
    It follows that $\sheaf F^i \otimes_{\sheaf R_X} \sheaf M$ is a direct sum of shifts of restrictions of $\sheaf M$ and hence is soft.
    As $\sheaf M$ is flat, we obtain an exact sequence
    \[
        \cdots \to \sheaf F^{-2} \otimes_{\sheaf R_X} \sheaf M \to \sheaf F^{-1} \otimes_{\sheaf R_X} \sheaf M \to \sheaf F^{0} \otimes_{\sheaf R_X} \sheaf M \to \sheaf F \otimes_{\sheaf R_X} \sheaf M \to 0,
    \]
    where each $\sheaf F^i \otimes_{\sheaf R_X} \sheaf M$ is soft.
    Thus Lemma~\ref{lem:coh_dim_and_soft_sequence} implies that $\sheaf F \otimes_{\sheaf R_X} \sheaf M$ is soft as well.
\end{proof}

\begin{Lemma}
    If $X$ has finite cohomological dimension, then the sheaf $\sheaf R$ has a finite resolution by soft and flat modules.
\end{Lemma}

\begin{proof}
    This is proven exactly as in the ungraded situation, see \cite[Proposition~VI.1.3]{Iversen:1986:CohomologyOfSheaves}.
    Note that this is where the assumption that $\sheaf R$ is noetherian is used (via \cite[Lemma~VI.1.4]{Iversen:1986:CohomologyOfSheaves})
\end{proof}

\begin{proof}[Proof of Theorem~\ref{thm:duality}]
    By Lemma~\ref{lem:duality_representability}, for any flat and soft $\sheaf R_X$-module $\sheaf M$ and any $\sheaf R_Y$-module $\sheaf G$ there exists a $\sheaf R_X$-module $f^!_{\gr,\sheaf M}(\sheaf G)$ and a canonical isomorphism
    \[
        \Hom_{\sheaf R_Y}(f_{\gr,!}(\sheaf F \otimes_{\sheaf R_X} \sheaf M), \sheaf G) \cong
        \Hom_{\sheaf R_X}(\sheaf F,\, f^!_{\gr,\sheaf M}(\sheaf G))
    \]
    for any $\sheaf R_X$-module $\sheaf F$.
    As the functor $f_{\gr,!}({-} \otimes_{\sheaf R_X} \sheaf M)$ is exact by the proof of Lemma~\ref{lem:duality_representability}, if $\sheaf G$ is injective, so is $f^!_{\gr,\sheaf M}(\sheaf G)$.
    From here one bootstraps up to the derived statement in the usual manner by taking $\sheaf M$ to be a finite soft and flat resolution of $\sheaf R_X$, see \cite[Theorem~VII.3.1]{Iversen:1986:CohomologyOfSheaves} or \cite[Theorem~3.1.5 and Proposition~3.1.10]{KashiwaraSchapira:1994:SheavesOnManifolds} for details.
\end{proof}

If $f\colon X \to Y$ and $g \colon Y \to Z$ are two morphisms of ringed graded topological spaces, then $\RR(g \circ f)_{\gr,!} \cong \RR g_{\gr,!} \circ \RR f_{\gr,!}$ and hence $\RR(g \circ f)_{\gr}^! \cong \RR f_{\gr}^! \circ \RR g_{\gr}^!$

Let $k$ be a commutative ring.
Recall that a dualizing complex for $k$ is a complex of $k$-modules $\omega_k \in \cat D_f^b(k)$ of finite injective dimension such that the canonical map $k \to \RHom_k(\omega_k,\, \omega_k)$ is an isomorphism \stackcite{0A7B}.
From now on we assume that $k$ has a dualizing complex $\omega_k$, which we fix \stackcite{0BFR}.
For example if $k$ is a field, one can take $\omega_k = k$.

\begin{Definition}
    Let $(X,\,\Lambda,\,\sheaf R)$ be a ringed graded topological space of finite cohomological dimension such that $\sheaf R$ is a graded sheaf of $k$-algebras.
    Let $p\colon X \to (\pt,\, 0,\, k)$ be the canonical map.
    We call $\dc_X = p_{\gr}^!\omega_k$ the \emph{dualizing complex} of $X$ and $\DD_X = \RsheafHom({-},\, \dc_X)$ the \emph{dualizing functor}.
\end{Definition}

\begin{Remark}\label{rem:how_to_compute_duality}
    Consider a ringed graded space $X = (X,\, \Lambda,\, \sheaf R_X)$ and let $\underline X = (X, 0, k_X)$ be the underlying topological space.
    Let $\pi\colon X \to \underline X$ be the canonical map.
    Then for any $\Lambda$-graded sheaf $\sheaf F$ on $X$ and any $\lambda \in \Lambda(X)$ one has $\pi_{\gr,*}(\sheaf F\langle \lambda\rangle) = \pi_{\gr,!}(\sheaf F\langle \lambda\rangle) = \sheaf F_\lambda$.
    Suppose we know the dualizing complex $\omega_{\ul X}$.
    Then,
    \begin{align*}
        (\omega_X)_\lambda & \cong
        \pi_{\gr,*}\RsheafHom_{\sheaf R_X}(\sheaf R_X\langle \lambda\rangle,\, \pi_{\gr}^!\omega_{\underline X}) \\ & \cong
        \RsheafHom_{k_X}(\sheaf R_{X,-\lambda},\, \omega_{\underline X}).
    \end{align*}
    Thus, knowing duality for $\underline X$, it is often not too hard to determine the dualizing complex for $X$.
\end{Remark}

\begin{Corollary}
    Let $f\colon X \to Y$ be a morphism of graded ringed spaces and assume that $X$ has finite cohomological dimension.
    Then:
    \begin{enumerate}
        \item \label{cor:duality_identities:1} $\RR f_{\gr}^! \RsheafHom_{\sheaf R_Y}(\sheaf F, \sheaf G) \cong \RsheafHom_{\sheaf R_X}(\LL f_{\gr}^*\sheaf F,\, f_{\gr}^!\sheaf G)$ for any $\sheaf F, \sheaf G \in \catDbSh{Y,\, \Lambda_Y,\, \sheaf R_Y}$.
        \item \label{cor:duality_identities:2} $\RR f_{\gr,*} \circ \DD_X \cong \DD_Y \circ \RR f_{\gr,!}$.
        \item \label{cor:duality_identities:3} $f_{\gr}^! \circ \DD_Y \cong \DD_X \circ \LL f_{\gr}^*$.
    \end{enumerate}
\end{Corollary}

\begin{proof}
    As in the classical case, \ref{cor:duality_identities:1} follows from Theorem~\ref{thm:duality}, tensor-hom adjunction and the projection formula (Proposition~\ref{prop:projection_formula}), see~\cite[Proposition~3.1.13]{KashiwaraSchapira:1994:SheavesOnManifolds}.
    Assertion \ref{cor:duality_identities:2} is immediate from Theorem~\ref{thm:duality} with $\sheaf G = \dc_Y$, while \ref{cor:duality_identities:3} follows from \ref{cor:duality_identities:1} in the same manner.
\end{proof}

\printbibliography

\end{document}